\documentclass[11pt,a4paper,reqno]{amsart}
\usepackage{amssymb}
\usepackage{amscd}
\usepackage{amsmath}
\usepackage{amsfonts}
\usepackage{mathrsfs}
\usepackage{algorithm}
\usepackage{cite}

\usepackage[small,nohug]{diagrams}
\diagramstyle[labelstyle=\scriptstyle]

\theoremstyle{plain}
\newtheorem{theorem}{Theorem}[section]

\newtheorem{lemma}[theorem]{Lemma}

\newtheorem{proposition}[theorem]{Proposition}

\theoremstyle{definition}

\def\Shape(#1){\operatorname{Shape}(#1)}

\newcommand{\magma}{{\sc Magma}}
\newcommand{\Z}{{\mathbb{Z}}}
\newcommand{\Q}{{\mathbb{Q}}}

\newcommand{\C}{{\mathbb{C}}}
\newcommand{\Prj}{{\mathbb{P}}}
\newcommand{\Cc}{{C_{can}}}
\newcommand{\gl}{\mathfrak{sl}}
\newcommand{\Osh}{{\mathcal{O}}}
\newcommand{\Cbe}{{\mathscr{C}}}
\newcommand{\res}{\mbox{\tt res}}

\begin{document}
\title{Explicit Solution By Radicals, Gonal Maps and Plane Models of Algebraic
Curves of Genus 5 or 6}
\author{Michael Harrison}

\begin{abstract}
We give explicit computational algorithms to construct minimal degree (always $\le 4$)
ramified covers of $\Prj^1$ for algebraic curves of genus 5 and 6. This completes the
work of Schicho and Sevilla (who dealt with the $g \le 4$ case) on constructing radical
parametrisations of arbitrary genus $g$ curves. Zariski showed that this is impossible for
the general curve of genus $\ge 7$. We also construct minimal degree birational plane
models and show how the existence of degree 6 plane models for genus 6 curves is related
to the gonality and geometric type of a certain auxiliary surface.
\end{abstract}
\maketitle

\section{Introduction}
\bigskip

A famous result of Zariski (see, eg, \cite{Fri89}) states that the general algebraic
curve $C$ over $\C$ of genus $> 6$ is not soluble by radicals over a rational function
field. That is, it is not possible to write a set of coordinate functions
of $C$ as radical expressions of a single parameter $x$ that is
itself a rational function on $C$: the function field of $C$, $\C(C)$, cannot be
expressed as a radical extension of a $\C(x)$ subfield.

In contrast to this, Brill-Noether theory tells us that a curve of genus
$g$ over an algebraically-closed field $k$ (any characteristic) has a map to $\Prj^1$
of degree $\le \lceil g/2 \rceil + 1$ \cite{KL74}. Thus the function field of
any $C$ of genus $\le 6$ is an extension of degree $\le 4$ of a rational function
field and so $C$ is soluble by radicals.

Once we have found $k(x)$ in $k(C)$ of index $\le 4$, the standard formulae for
the roots of a polynomial of degree $\le 4$ give radical expressions for the
generators of $k(C)$. The main computational problem is the construction
of the smallest degree $d$ map from $C$ to $\Prj^1$ ($d$ is traditionally referred to
as the {\it gonality} of $C$).

For genus $\le 4$, when $d \le 3$, explicit
algorithms for this have been described in \cite{ScSe09}. Here, we will fill the gap by
giving explicit algorithms for the genus $5$ and $6$ cases. When $C$ is not
hyperelliptic, the algorithms are based on
the theory of rational scrolls containing the canonical embedding $\Cc$ of $C$
as described in \cite{Sch86}. The case that requires the most work to produce an
algorithm from the general theory is for the general curve of genus 6 (with gonality
4). There is also the added complication of distinguishing between the
gonality 3 case and the $g^2_5$ case (gonality 4) for genus 6. The general genus 5 case
(gonality 4 also) is fairly straightforward both theoretically and computationally.
We include it for completeness. The majority of the paper, though, deals with genus 6.

The construction of minimal degree covers of $\Prj^1$ (we sometimes refer to these as
{\it gonal} maps) is an interesting and important computational problem in its own right,
with applications outside of radical parametrisation. For example, it is useful to
be able to express the function field of a curve as a finite extension
of the rational function field $k(x)$ of as small a degree as possible when applying
function field algorithms like those of Florian Hess or Mark van Hoeij to the
computation of Riemann-Roch spaces and the like. Of course, we assume that such algorithms
are available for some of the operations on our initial curve here. However, the possibility
of reexpressing the curve as a smaller degree extension of $k(x)$ for more general analysis
is a valuable one.

Of related interest is the construction of a smallest degree birational plane model of the
curve in $\Prj^2$ (singular, in general). Fortunately, in a number of cases, these can be
constructed directly from the same data generated from the algorithm to construct gonal
maps. In particular, this is true in the generic genus 6 case. In addition, in the genus
6, 4-gonal case, we will show that there are only finitely many gonal maps up to equivalence
precisely when the curve has a degree 6 plane model. Because of these relations and the intrinsic
interest, we also show how our methods allow us to construct these minimal degree models
and analyse the question of the existence of plane models of certain degrees in some detail.

The contents of the paper are as follows. In the next section, we introduce the general
setup and discuss the 3-gonal cases and the plane quintic genus 6 case. The third section
gives the algorithm for 4-gonal, genus 5 curves. The fourth section covers the general
(Clifford-index 2) genus 6 case. There, we review Schreyer's analysis of the minimal
free polynomial resolution of the canonical coordinate ring, present the algorithms
for finding gonal maps and plane models and, finally, prove some results about the
existence of degree 6 plane models. We also consider the example of the genus 6 modular
curve $X_0(58)$, using our algorithm to construct a degree 4 rational function, 
a degree 6 plane model and a radical parametrisation.

The algorithms are applied to $C$ defined over a non-algebraically closed base field
$k$ and construct degree $d$ maps to $\Prj^1$ over a finite extension of $k$. In general,
no degree $d$ map exists over $k$.

I would like to thank Josef Schicho and David Sevilla for introducing me to the problem
of computing radical parametrisations.
\bigskip

\section{Generalities and the non-generic cases}
\bigskip
The curve $C$ of geometric genus $g$ is assumed to be initially given in a general 
geometric realisation: 
by a birational (maybe singular) model in affine or projective space of arbitrary
dimension or as an algebraic function field. There are known algorithms to compute
the canonical image $\Cc$ in $\Prj^{g-1}$ along with an explicit birational map from
$C$ to $\Cc$ or an isomorphism of their function fields. For example, the function field
methods of Hess \cite{Hes02}. We also need to explicitly compute the minimal free
resolution of the defining ideal $I_\Cc$ of $\Cc$ as a graded $R$-module, where
$R = k[x_1,\ldots,x_g]$ is the coordinate ring of $\Prj^{g-1}$. There are also well-known
algorithms for this, computing syzygies via Gr\"obner bases.

Throughout the paper, we use the language of linear systems. $g^r_d$ will denote a
(not necessarily complete) linear system on $C$ of degree $d$ and projective dimension $r$.
Such a system gives a morphism of $C$ into $\Prj^r$ up to a linear change of coordinates,
as described in Chapter II, Section 7 of \cite{H77}.
If the $g^r_d$ is basepoint free, this map will be a finite map of degree
$d$ onto its image. If the system has a base locus of degree $e$, then the map will be of
degree $d-e$. The image lies in no hyperplane. Thus, a degree $d$ cover of $\Prj^1$,
up to automorphisms of $\Prj^1$, is equivalent to a basepoint free $g^1_d$ on $C$.
This is also equivalent to a rational function of degree $d$ on $C$ up to equivalence,
where $f$ and $g$ are equivalent if $g = (af+b)/(cf+d)$ for a non-singular matrix
$\left({a \atop c}{b\atop d}\right)$.
 
The results in \cite{Sch86} (or \cite{Eis05}) tell us that a
$g_d^1$
on $C$ will give rise to a rational scroll $X$ in $\Prj^{g-1}$ of dimension $d-1$
containing $\Cc$ such that the ruling of $X$ cuts out the $g^1_d$ on $\Cc$.
Furthermore, the minimal free resolution of $I_X$, the defining ideal of $X$,
occurs as direct summand of the {\it quadratic strand} of the minimal free
resolution of $I_\Cc$ (see Section 6C and Appendix A2H of\cite{Eis05}). Given $X$,
the map to the $\Prj^1$ base of the ruling, which induces the degree $d$ map on $\Cc$,
can be computed in a number of ways - the Lie algebra method, for example
\cite{ScSe09}. When the scroll is of codimension 2 in $\Prj^{g-1}$, the map can be
read off directly from the resolution of $I_X$. This occurs for a $g^1_3$ for
genus 5 or a $g^1_4$ for genus 6.
\bigskip

The problem of constructing gonal maps can thus be reduced to explicitly constructing direct sum 
subcomplexes of the canonical resolution $F$ for $I_\Cc$ that will give the resolution of
such a scroll. The gonality $d$ of $C$ can be read off from the shape of $F$
(\cite{Sch86}). The hyperelliptic case ($d = 2$) can be recognised by the
arithmetic genus of the canonical image $\Cc$ being zero. Then, $\Cc$ is a rational
normal curve, $C \rightarrow \Cc$ is of degree 2, and a parametrisation of $\Cc$
can be constructed in a number of ways: the Lie algebra method again, 
repeated adjunction mappings or using function field methods (\cite{Hes02}).

The classical $d=3$ case (Petri's Theorem) is dealt with in \cite{ScSe09} using
the Lie algebra method. This is characterised
by $I_\Cc$ having a minimal basis that includes cubic forms as well as quadrics
(for $g \ge 4$). The minimal resolution of $I_X$ is then given by the full quadratic
strand of the canonical resolution. No extra computation is required to find $X$ :
$I_X$ is just the ideal generated by the space of quadric forms in $I_\Cc$. There
is the extra complication in the genus 6 case where $X$ may be the Veronese surface
in $\Prj^5$ rather than a rational scroll. Here $C$ is isomorphic to a non-singular plane
curve of degree 5 and the gonality is 4 rather than 3. In our algorithm, this
possibility is discovered when the Lie algebra method is applied to $X$ giving
the sub-Lie algebra of trace 0 6x6 matrices corresponding to the closed subgroup of
automorphisms of $\Prj^5$ that leave $X$ invariant. When $X$ is Veronese, this Lie
algebra is simple of dimension 8, isomorphic to $\gl_3$ over $\bar{k}$. In the scroll
cases, the algebra is of dimension 7 or 9 or is soluble of dimension 8.

When $X$ is a Veronese surface, the Lie algebra method gives an isomorphism $\phi$ 
of $X$  to $\Prj^2$ over a finite extension of $k$. This comes from computing a linear isomorphism
from $X$ onto the standard Veronese surface $X_0$ in $\Prj^5$ by determining a linear
isomorphism within $\gl_6$ taking the Lie algebra of $X$ to that of $X_0$.
Under $\phi$, $\Cc$ is mapped to a non-singular degree 5 plane curve $C_5$. Then, the
degree 4 maps to $\Prj^1$ are precisely the projections from points on $C_5$.

This just leaves the generic cases for genus 5 and 6 curves where the Betti diagrams
for the minimal free resolutions of the canonical coordinate rings are as follows
(\cite{Sch86}).
\smallskip

\makebox[15cm][s]{
  \begin{tabular}{c|cccc}
   & 0 & 1 & 2 & 3 \\ \hline
 0 & 1 & - & - & - \\
 1 & - & 3 & - & - \\
 2 & - & - & 3 & - \\
 3 & - & - & - & 1 \\
%   & & & g=5 & &  \\
  \end{tabular}

  \begin{tabular}{c|ccccc}
   & 0 & 1 & 2 & 3 & 4 \\ \hline
 0 & 1 & - & - & - & - \\
 1 & - & 6 & 5 & - & - \\
 2 & - & - & 5 & 6 & - \\
 3 & - & - & - & - & 1 \\
%   & & & g=6 & &  \\
  \end{tabular}
}
\smallskip

\noindent The algorithms for these cases will be given in the two following sections.
\medskip

We should note here some situations when it is easy to find gonal maps. If we have a
plane model of $C$ of degree $d$ and $P$ is a point on $C$ of multiplicity $m$, then
projection from $P$ gives a degree $d-m$ map to $\Prj^1$. A small degree (singular)
plane model with a singular point of sufficiently high multiplicity can, therefore,
give us gonal maps by simple projection.

Assume that we have computed a gonal map for $C$ which gives us a rational function $t$
on $C$ of degree $\le 4$. We can then compute a radical parametrisation as follows.
Firstly, we choose a second rational function $x$ that generates the function field
$\bar{k}(C)$ over $\bar{k}(t)$, preferably a coordinate function of $C$ in an 
affinisation of its given algebraic representation (N.B.: $\bar{k}(C)/\bar{k}(t)$ is
separable for gonal maps). We then express all coordinate functions on $C$ as
rational functions in $x$ and $t$. This can be achieved with standard elimination
techniques using Gr\"obner bases or resultants. Finally, we find a radical expression
for $x$ in terms of $t$ using the minimal polynomial of $x$ over $\bar{k}(t)$.
This polynomial is of degree $\le 4$ and we can just use the standard formulae for
its roots as radical expressions in the coefficients.

\section{Genus 5}
\bigskip

In the generic (no $g^1_d$ for $d \le 3$) case, the $g^1_4$s correspond to degree 2,
dimension 3 rational scrolls in $\Prj^4$ containing $\Cc$ ((6.2) \cite{Sch86}).
These are defined by singular quadrics $Q$ in the 3-dimensional space of quadrics
in $I_\Cc$.

Such singular quadrics are given by 5-variable quadric forms of rank 4 or 3 (no
quadrics of rank $<= 2$ can occur in $I_\Cc$ since such quadrics are geometrically
reducible and $\Cc$ isn't contained in any hyperplane). These correspond to rational
scrolls of type $S(1,1,0)$ or $S(2,0,0)$ in Schreyer's notation.

The first type are cones over non-singular quadric surfaces $S$ in $\Prj^3$. The
projection(s) to $\Prj^1$ are given by projection from the unique
singular point followed by one of the 2 classes of fibre projection from $S$ to
$\Prj^1$.

The second type are cones over non-singular conics $C_0$ in $\Prj^2$. For these,
the projections to $\Prj^1$ are given by projection from the singular locus
(a line) followed by a birational map of $C_0$ to $\Prj^1$.

In either case, the projections may be defined over a quadratic extension of the
base field $k$ and it is well-known how to compute them explicitly.
\bigskip

It remains to find such singular quadrics and this is also computationally fairly
straightforward.

If $Q_1,Q_2,Q_3$ form a basis for the space of quadrics in $I_\Cc$, we can write
the general quadric as $Q_{x,y,z} = xQ_1+yQ_2+zQ_3$ for variable $x,y,z$. The
condition that $Q_{x,y,z}$ is singular leads to a single degree 5 homogeneous equation
$F(x,y,z) = 0$ as follows.

The 5 partial derivatives of $\partial Q_{x,y,z}/\partial x_i$ with respect to the
5 coordinate variables $x_i$ of $\Prj^4$ are linear forms in the $x_i$ with coefficients
given by linear forms in $x,y,z$. The condition that all partial derivatives vanish
at a point of $\Prj^4$ is that the 5x5 matrix of coefficients of the partials w.r.t
the $x_i$ is singular: ie, that its determinant vanishes. $F(x,y,z)$ is just given
by this determinant.

Thus we explicitly find a one-dimensional projective family of singular quadrics.
This accords with Brill-Noether theory which tells us that the dimension of the family of
$g^1_4$s is at least one. We can now choose a non-trivial solution of $F(x,y,z) = 0$ over a
finite extension of $k$ of degree $\le 5$. In general, $F$ will be non-singular and
irreducible, as is easily checked in random examples. It is a number-theoretically very
difficult question to check whether a solution exists over $k$ when $k=\Q$.
\bigskip

\noindent{\it Plane models.} If $C$ is birationally equivalent to a plane curve of degree
$d$ in $\Prj^2$, $d\ge 5$ by the arithmetic genus formula. A plane model of degree 5 exists if
and only if $C$ has a basepoint-free $g^2_5$. If $K$ denotes the canonical divisor class on
$C$, then $g^2_5$s correspond to $g^1_3$s by Riemann-Roch under $D \leftrightarrow K-D$.
So, for a plane model of degree 5, the gonality of $C$ must be less than 4. A degree 5
plane model would have to have a node or cusp as the unique singularity for its normalisation
$C$ to have genus 5. It is then easy to see that the canonical linear system would separate
points on $C$ because, by adjunction, it is given by the system of plane quadrics passing 
through the singular point. Hence, $C$ cannot be hyperelliptic so must have gonality exactly
3.

Conversely, if $C$ has gonality 3, it has a degree 5 plane model.
$\Cc$ is contained in a rational scroll $X$ which is
isomorphic to the Hirzebruch surface $X_1$ embedded in $\Prj^4$ via the very ample
divisor $C_0+2f$, in the notation of Chapter V, Section 2 of \cite{H77}. $\Cc$ has divisor
class $3C_0+5f$ on $X_1$. Contraction of the unique (-1)-curve $C_0$ on $X_1$ gives a
birational morphism of $X_1$ onto $\Prj^2$ that maps $\Cc$ birationally onto a degree
5 curve (the self-intersection of the image of $\Cc$ is 25).

Computationally, a rational map from $X$ to $\Prj^2$ that contracts $C_0$ can be constructed
by projecting from any line in the ruling of $X$. Lines
in the ruling of $X$ can be located by the Lie algebra method again, just as the fibration
$X \rightarrow \Prj^1$ is computed in \cite{ScSe09}.
\medskip

It is generally easy to find a degree 6 model for 4-gonal genus 5 curve.
If $D$ is an effective divisor of degree 6 on $C$, then, by Riemann-Roch, dim $|D| \ge 2$ if
and only if $|K-D|$ is
non-empty, i.e., $D$ is the complement of a subcanonical degree 2 divisor. 
Any effective divisor of degree 2 is subcanonical. If $C$ is
not hyperelliptic, then we can take $D$ equal to the canonical complement of any two distinct 
points, $P,Q$ on $C$. $|D| = |K-(P+Q)|$ is given by the hyperplane sections of $\Cc$ in
$\Prj^4$ that contain the line $L$ joining the images of $P$ and $Q$. So dim $|D|$ is exactly
2 and the rational map $C \rightarrow \Prj^2$ is just given by projection from $L$. If this
rational map is not birational on $C$, then it must give a 2-1 map onto a smooth cubic
(if it were 3-1 onto a conic or the cubic was singular then $C$ would have gonality $\le 3$),
so $C$ is a 2-1 cover of a genus 1 curve $E$.

In this special case, $E$ is the unique genus
1 curve that $C$ is a double cover of ({\it c.f.} Remark 2 of Section~\ref{ell_cone_case}).
I think that, in this case, all $g^1_4$s on $C$ are pullbacks of $g^1_2$s on $E$. This is true
in the genus 6 case - as I show in the next section - but I haven't checked it properly for genus 5.
It would imply that $D=K-(P+Q)$ gives a birational map from $C$ into $\Prj^2$ if
$P$ and $Q$ don't lie over the same point in $E$. This is because $\mbox{dim}|D-R| = 1+
\mbox{dim}|P+Q+R|$ and $\mbox{dim}|D-R-S| = \mbox{dim}|P+Q+R+S|$. If $R$ doesn't lie over
either of the images of $P$ or $Q$ in $E$, $|P+Q+R+S|$ cannot be the pullback of a $g^1_2$
for any $S$, so must have dimension zero. Then $\mbox{dim}|D-R| = 1$ and $ \mbox{dim}|D-R-S| = 0$
for any $S$. Thus, the linear system $|D|$ separates $R$ from other points (and
tangent vectors at $R$).

When $C$ is hyperelliptic, the same argument shows that any $|D|$ gives a map to $\Prj^2$
that factors through the canonical projection $C\rightarrow\Prj^1$, so cannot be birational.
When $R$ is not a Weierstrass point and is distinct from $P,Q$ and their images under the
hyperelliptic involution $w$ ($P$ could be $Q$ here), $\mbox{dim}|P+Q+R+w(R)| = 1+
\mbox{dim}|P+Q+R|$. This means that $\mbox{dim}|D-R|=\mbox{dim}|D-R-w(R)|$.

\section{Genus 6 (generic)}
\bigskip

This is the most difficult case. There are $g^1_4$s but no $g^1_d$ for $d < 4$ and no $g^2_5$
(equivalently, $C$ has Clifford-index 2).
The rational scrolls $X$ to be constructed, which contain $\Cc$ and induce a $g^1_4$ on it via the ruling,
are of degree 3 and dimension 3 in $\Prj^5$ (the points in any divisor of the $g^1_4$ span
a linear space of dimension 2 in the canonical embedding by geometric Riemann-Roch, so the dimension
must be 3, and the degree is the codimension plus one).

Throughout this section, $R$ will
denote the coordinate ring $k[x_1,\ldots,x_6]$ of the $\Prj^5$ containing $\Cc$ and we
will write $I_Z$ for the saturated ideal of $R$ that corresponds to a closed subscheme
$Z$ of $\Prj^5$. As usual, $R(n)$, $n \in \Z$, will denote the
graded $R$-module equal to $R$ as a plain $R$-module but having a shift in the
grading so that $R(n)_m = R_{n+m}, m \in \Z$. All $R$-module homomorphisms between
graded $R$-modules preserve the grading. Elements of modules of the form
$R(a_1)^{m_1}\oplus\ldots\oplus R(a_r)^{m_r}$ will be thought of as row vectors of
length $m_1+\ldots +m_r$ with entries in $R$. Thus, a homomorphism from
$R(a_1)\oplus\ldots\oplus R(a_r)$ to $R(b_1)\oplus\ldots\oplus R(b_s)$ will be
represented by an $r$-by-$s$ $R$-matrix whose $(i,j)$-th entry will be a homogeneous polynomial
of degree $b_j-a_i$ (necessarily zero, if $a_i > b_j$) acting on row vectors by right 
multiplication.

The minimal resolution of $I_X$ as a graded $R$-module is of the form
$R(-2)^3 \leftarrow R(-3)^2 \leftarrow 0$. This complex
must occur as a graded direct summand (see Prop. 6.13 of \cite{Eis05}) of a minimal
resolution of $R/I_\Cc$
$$\res:\ \ R \stackrel{\psi}{\longleftarrow} R(-2)^6 \stackrel{\phi}{\longleftarrow} R(-3)^5\oplus R(-4)^5
   \leftarrow R(-5)^6 \leftarrow R(-6) \leftarrow 0 $$
the $R(-2)^3$ being a direct summand of $R(-2)^6$ and the $R(-3)^2$ a direct summand
of $R(-3)^5$.

We fix a particular such resolution and \res\ will refer to it throughout this section.
\medskip

Firstly, we have the following 
\begin{lemma} There is a 1-1 equivalence between $g^1_4$s on $C$ and the degree 3,
dimension 3 rational scrolls in $\Prj^5$ containing $\Cc$.
\end{lemma}

\begin{proof} A $g^1_4$ gives a unique rational scroll of the correct type that is
 the union of the degree 2 linear spans of the $D \in g^1_4$
on $\Cc$. Note that the $g^1_4$ pencil is a {\it complete, basepoint-free} linear
system as otherwise
there would exist an $E < D \in g^1_4$ giving a $g^1_d$ for $d < 4$. This
would contradict the gonality 4 assumption on $C$.

On the other hand a degree 3, dimension 3 rational scroll containing $\Cc$ cuts out
a $g_d^1$ with dim $|K-D| = 2$ for $D \in g^1_d$ (Section 2,\cite{Sch86}).
$K$ denotes the canonical divisor class on $C$. It remains
to show that $d$ must be 4 here. But $d$ must be at least 4, as $C$ has gonality 4.
deg($K-D$) is $10-d$, so if $d \ge 5$ then we would have a $g^2_5$ or a $g^2_e$ for
$e \le 4$. The latter would lead to a $g^1_f$ for $f < 4$. Neither of these is possible by
assumption.
\end{proof}
\bigskip

Our approach here is to explicitly compute such scrolls by a direct, brute-force method.
This does appear to work quite efficiently in practise. An alternative would be
to try to work with the easily computable surface $Y$ discussed in the next section.
The next section contains a summary of more detailed structure of \res\ that we need to
refer to. Our algorithm will be given in the following section.
 
\subsection{The Canonical Resolution of $\Cc$}
\label{sec_can_res}
\bigskip

Fixing a $g^1_4$ on $C$ (which exist by \cite{KL74}) and corresponding rational
scroll $X$ containing $\Cc$, Sections (6.2)-(6.5) of \cite{Sch86} give the following more 
detailed description of the way that the canonical resolution \res\ is built up.
\medskip

$\Cc$ is the complete intersection of two divisors $Y$ and $Z$ of $X$. The surfaces $Y$ and $Z$ 
are (absolutely) irreducible since $\Cc$ is. Writing $H$ for the
hyperplane class on $X$ and $R$ for the class of the ruling, we have the follwing equivalence of 
rational divisor classes on $X$
$$ Y\sim 2H-R\qquad Z\sim 2H$$
In Schreyer's notation, $f=3$, $b_1=1$ and $b_2=0$ here. The case $(b_1,b_2)=(2,-1)$ cannot
occur, as it would lead to $C$ having a $g^1_3$ or $g^2_5$ as Schreyer notes in his (6.3).

In fact, if $Y\sim 2H-2R$, $Y$ would be a degree 4 ($\deg(Y) = 6-b_1$) surface in $\Prj^5$ and
so would be a Veronese surface or a 2-dimensional rational scroll. The Veronese case would
give a $g^2_5$, an isomorphism of $Y$ to $\Prj^2$ mapping $\Cc$ isomorphically
to a non-singular degree 5 plane curve. In the scroll case, $Y$ would have to be a Hirzebruch
surface $X_0$, $X_2$ or $X_4$ embedded in $\Prj^5$ by the invertible sheaves with 
divisor class $H_1 = C_0+2f, C_0+3f$ or $C_0+4f$ respectively, in the notation of Chapter V, 
Section 2 of \cite{H77} ($S(2,2), S(3,1)$ or $S(4,0)$ in Schreyer's notation). $\Cc$ pulls
back to a smooth genus 6 curve which has intersection 10 with $H_1$. Thus, the divisor class of 
this pullback would be $3C_0+mf$ with $m$ equal to $4,7$ or $10$, and the ruling of $Y$ would 
induce a $g^1_3$ on $C$.
\medskip

Sections (6.4) and (6.5) of \cite{Sch86} give two possibilities for $Y$ depending upon whether 
the generic fibre over $\Prj^1$ is an irreducible conic or not.

In the general case (6.4), $Y$ is a (possibly singular) degree 5 Del Pezzo surface,
anticanonically embedded in $\Prj^5$. In (6.4), $k=1$ and $\delta=3$, so $Y$ is the Hirzebruch
surface $X_1$ blown up at 3 points and $X_1$ is itself $\Prj^2$ blown up at a point.

In case (6.5), $Y$ is a (singular) cone over a projective normal elliptic curve $E$ of a
hyperplane of $\Prj^5$.

These possibilities also follow from the classification of surfaces of degree $d$ in $\Prj^d$
in \cite{Nag60} and the fact that $Y$ is arithmetically Cohen-Macaulay, so linearly normal and
not a projection from $\Prj^6$ (the minimal free resolution of its defining ideal $I_Y$ is given below).
\medskip

Replacing coherent sheaves $\mathcal{F}$ on $\Prj^5$ by their corresponding maximal graded $R$-modules
($\mathcal{F} \leadsto \oplus^\infty_{n=-\infty}H^0(\Prj^5,\mathcal{F}(n))$, so 
$\mathcal{O}_{\Prj^5}(n)$ gives $R(n)$, $\mathcal{O}_\Cc$ gives $R/I_\Cc$ etc.), Schreyer shows that
Serre twists of the Buchsbaum-Eisenbud exact sequences (see Section 1, \cite{Sch86}),
$\Cbe^0$ and $\Cbe^1$
$$ \Cbe^0:\qquad R \leftarrow R(-2)^3 \leftarrow R(-3)^2 \leftarrow 0$$
$$ \Cbe^1:\qquad R^2 \leftarrow R(-1)^3 \leftarrow R(-3) \leftarrow 0$$
determine the minimal free resolutions (up to isomorphism) of the coordinate rings of $\Cc$, $Y$ and
$Z$ in the following way.

$\Cbe^0,\Cbe^0(-2),\Cbe^1(-2),\Cbe^1(-4)$ are respectively minimal free resolutions of the maximal graded
modules corresponding to the sheaves $\Osh_X,\Osh_X(-Y)$\footnote{This isn't quite correct
if $Y$ intersects the singular locus of $X$ and is non-Cartier. $\Osh_X(-Y)$ then actually
means the pushforward of the corresponding sheaf on the projective bundle of which $X$ is 
an image.},$\Osh_X(-Z),\Osh_X(R-4H)$.

The morphisms of the Koszul-complex resolution of $\Osh_\Cc$ by locally free $\Osh_X$ sheaves
$$ \Osh_X \leftarrow \Osh_X(-Y)\oplus\Osh_X(-Z) \leftarrow \Osh_X(R-4H)\leftarrow 0 $$
extends to morphisms between complexes of graded modules $\Cbe^1(-4) \stackrel{(f1,f2)}{\longrightarrow}
\Cbe^0(-2)\oplus\Cbe^1(-2)$ and $\Cbe^0(-2)\stackrel{g1}{\rightarrow}
\Cbe^0$, $\Cbe^1(-2)\stackrel{g2}{\rightarrow}\Cbe^0$ such that the mapping cones of the latter two
give minimal free resolutions of the coordinate rings of $Z$ and $Y$ respectively and the iterated
mapping cone of
$$\left[\Cbe^1(-4) \stackrel{(f1,f2)}{\longrightarrow}\Cbe^0(-2)\oplus\Cbe^1(-2)\right]
    \stackrel{g1+g2}{\longrightarrow}\Cbe^0$$
gives a minimal free resolution of $R/I_\Cc$.

Explicitly we get
$$\begin{diagram}[midshaft]
R(-7)&\rTo^{(\alpha_3,\beta_3)}&R(-5)^2\oplus R(-5)&\rTo^{\chi_3+\phi_3}&R(-3)^2\\
&&\dTo<{\delta_2\oplus \epsilon_2}&\ruTo_{\sigma_3}&\\
\dTo<{\rho_2}&&&&\dTo<{\gamma_2}\\
&\ruLine&&&\\
R(-5)^3&\rTo^{(\alpha_2,\beta_2)}&R(-4)^3\oplus R(-3)^3&\rTo^{\chi_2+\phi_2}&R(-2)^3\\
&&\dTo<{\delta_1\oplus \epsilon_1}&\ruTo_{\sigma_2}\\
\dTo<{\rho_1}&&&&\dTo<{\gamma_1}&\\
&\ruLine&&&\\
R(-4)^2&\rTo^{(\alpha_1,\beta_1)}&R(-2)\oplus R(-2)^2&\rTo^{\chi_1+\phi_1}&R\\
\end{diagram}$$
The columns are the twisted $\Cbe$ exact sequences, the middle one being the direct sum of
$\Cbe^0(-2)$ and $\Cbe^1(-2)$. $f_1,f_2,g_1,g_2$ are given by the $\alpha,\beta,\chi$ and $\phi$
maps. The $\sigma$ maps are those that arise from iterating the mapping cone. Apart from these
diagonal maps, the other maps in the diagram give commutative squares.
\medskip

\noindent\underline{Uniqueness of $Y$}
\medskip

A resolution of $R/I_Y$ is given by the mapping cone of the morphism between complexes defined by the
$\phi$ maps. $I_Y \subset I_\Cc$ is the sum of the images of $\phi_1$ and $\gamma_1$ in $R$ and
is resolved by the direct summand subcomplex arising from the $\Cbe^1(-2)$ and $\Cbe^0$ terms
(apart from $R$). Our chosen \res\ is isomorphic to this iterated mapping cone as a complex of
graded $R$-modules, but different choices of $X$ will give different decompositions of it.

However, the direct summand subresolution of $I_Y$ is {\it always the same} and so we get the same {\it 
unique} $Y$ for any $X$ (this was noted by Schreyer in \cite{Sch86}).

The point is that in the resolution of $R/I_Y$
\begin{equation}
R \stackrel{-\phi_1+\gamma_1}{\longleftarrow} R(-2)^3\oplus R(-2)^3 \stackrel{(\epsilon_1\oplus 0,
  \phi_2\oplus\gamma_2)}
 {\longleftarrow} R(-3)^3\oplus R(-3)^2 \stackrel{(\epsilon_2,-\phi_3)}{\longleftarrow} R(-5)
 \leftarrow 0 \label{Y_res}
\end{equation}
the weight -3 part must correspond precisely to the $R(-3)^5$ summand in the third term of \res\ 
and the -2 part must then correspond to the unique dimension 5 summand $F_5$ of the $R(-2)^6$ second term
that contains the image of $R(-3)^5$. $I_Y$ is generated by the 5 quadrics which are the images
of the generators of $F_5$ under $\psi$. $Z$, which varies with $X$, is the intersection of $X$ with
a single quadric hypersurface and $\Cc$ is the intersection of that hypersurface with $Y$.
\medskip

The computational approach taken here is to actually compute scrolls and this seems to be
fairly efficient, algorithmically. However, it may be possible to just work with $Y$, which
is very easy to determine from \res\ by computing $F_5$. When $Y$ is a cone over an elliptic curve $E$,
there are an infinite number of $g^1_4$ on $C$, given by projecting to $E$ (a degree 2 map on $\Cc$) and
then taking any degree 2 map to $\Prj^1$. We do use this projection when we recognise that we are in this
case. In the general case, $Y$ is a degree 5 Del Pezzo and an alternative is to try to parametrise $Y$
by plane cubics over an extension of the base field. The inverse parametrisation gives a birational map
of $Y$ to a degree 6 plane curve and we can then just project from a singular point. We will briefly
return to this later. In any case, $F_5$ is used in the algorithm for computing $X$.

%$X$ is a scroll of type $S(1,1,1)$ or $S(2,1,0)$ in Schreyer's notation. Let $\Prj(\mathcal{E})$ be
%the projection bundle over $\Prj^1$ which it is an image of under birational map $f$. The sheafified 
%version of the minimal
%free resolution of $R/I_\Cc$ (replacing $R(n)$ with $\mathcal{O}_{\Prj^5}(n)$) a free resolution of
%$\mathcal{O}_\Cc$ by direct sums of Serre twists of the structure sheaf of $\Prj^5$, comes from
%pushing forward under $f$ a 
\bigskip

\subsection{The Algorithm}
\bigskip

Our method consists of constructing scrolls $X$ by determining the graded direct summand subcomplexes
of \res\ that resolve their defining ideals. For this purpose, we want computational conditions that are 
as simple as possible for characterising these subcomplexes. One possible potential problem is that the
restriction of $\phi$ to the $R(-3)^5$ term in \res\ has a nontrivial kernel that must be avoided.
The next proposition, however, shows that we don't have to worry about this and that only the most
obvious restriction needs to be checked. This reduces the computation to a reasonably
straightforward Grassmannian-type algebra problem.
\medskip

\begin{proposition}\label{FG_prop}
If $F$ is any rank 3 direct summand of the $R(-2)^6$ term of \res\ and
$G$ is any rank 2 direct summand of the $R(-3)^5$ term such that $G$ maps into $F$ under $\phi$,
then $F \leftarrow G \leftarrow 0$ is a minimal free resolution of $I_X$, the image
of $F$ in $R$ under $\psi$, which is the (maximal) defining ideal of a degree 3, dimension 3 scroll
$X$ containing $\Cc$.
\end{proposition}

\begin{proof} The key result that is needed is the
\medskip

\noindent{\it Claim:} For any $F$ and $G$ as in the statement of the proposition, $G \rightarrow F$
is injective.
\medskip

We will prove this below. We now show that the proposition follows from the claim. 

Let $I_X$ denote the ideal image of $F$ under $\psi$. $I_X$ is generated by three quadric forms,
$Q_1$, $Q_2$, $Q_3$ which are $\bar{k}$-linearly independent in the space of all quadric (degree 2)
forms in $R$, where $\psi : F \rightarrow R$ is given by the $1\times 3$ matrix $[Q_1\ Q_2\ Q_3]$.

We want to show that the complex
$$ 0 \leftarrow R/I_X \leftarrow R \stackrel{\psi}{\longleftarrow} F \stackrel{\phi}{\longleftarrow} G \leftarrow 0 $$
is exact. $\phi$ here is given by a $2\times 3$ matrix $M$ which can be written as
$$ \left(\begin{array}{ccc}\lambda_1 & \lambda_2 & \lambda_3\\ \mu_1 & \mu_2 & \mu_3\end{array}\right)$$
where the $\lambda_i, \mu_i$ are linear forms in $R$.

Exactness follows from the Hilbert-Burch theorem (Thm. 20.15, \cite{Eis94}) once we show that there exists a
non-zero constant $a$ such that the three minors $Q'_i$ of $M$
($Q'_i$ is the minor obtained by leaving out the {\it i}th column of $M$) satisfy $Q'_i = (-1)^iaQ_i$,
$1 \le i \le 3$, and that the ideal $I_X$ (which they then generate) has depth $\ge 2$.

First note that each $Q_i$ is irreducible in $R\otimes_k\bar{k}$ since if one of them 
decomposed into a
product of 2 linear forms then $\Cc$ would lie in a hyperplane, which it doesn't. As $R$ is a UFD,
the principal ideal $(Q_1)$ is therefore prime and, as the $Q_i$ are $k$-linearly independent, $Q_2$ (or
$Q_3$) doesn't lie in $(Q_1)$ and $I_X$ has depth $\ge 2$.
\medskip

Now, if $E$ denotes the field of fractions of $R$, then $M$ has rank 2 over $E$, as $\phi$ is injective
on $G$,
so the $Q'_i$ are not all zero and the right kernel of $M$ in $E^3$ is generated by the column vector
${\bf v'} := (Q'_1, -Q'_2, Q'_3)^t$. From our complex, ${\bf v} := (Q_1, Q_2, Q_3)^t$ lies in the $R^3$-kernel.

The $Q'_i$ are quadrics. Let $b$ be a GCD of them in $R$ and write $U_i$ for $Q'_i/b$. Then the right
$R^3$-kernel of $M$ is free of rank 1 over $R$ generated by ${\bf w} := (U_1, U_2, U_3)^t$. Let the degree of $b$
be $r$. Then ${\bf v}$ is equal to $a{\bf w}$ with a non-zero homogeneous polynomial $a$ of degree
$r$. This implies that $r = 0$ and that $a,b$ are constants, since GCD~$(Q_1, Q_2, Q_3) = 1$. So
${\bf v'} = (a/b){\bf v}$, as required.
\medskip

From the minimal free $R$-resolution of $S := R/I_X$, it is immediate that $S$ has Hilbert series 
$(1+2t)/(1-t)^4$ and Hilbert polynomial $H(n) = (1/2)(n+2)(n+1)^2$, so the
closed subscheme $X$ of $\Prj^5$ corresponding to $I_X$ has degree 3 and dimension 3.

We claim that $M$ is 1-generic (over $\bar{k}$), as defined in Chapter 6 of \cite{Eis05}. Any non-zero 
quadric in the $\bar{k}$-linear span $\langle Q_1, Q_2, Q_3\rangle$ of its minors is
irreducible over $\bar{k}$, since no irreducible quadric vanishes on $\Cc$ as noted above.
If $M$ were not 1-generic, we could find invertible 2x2 and 3x3 matrices $U,V$ with
coefficients in $\bar{k}$ giving $UMV$ a zero in the top left position. But then at least two
of the minors of $UMV$ would be reducible and they are non-zero linear combinations of the 
linearly independent $Q_i$.

Then, Theorems 6.4 and A2.64 of \cite{Eis05} show that $X$ is a rational scroll.
\bigskip

\noindent{\it Proof of claim} 

It remains to prove that $G \rightarrow F$ must be injective. This follows from the more 
detailed description of \res\ given in Section \ref{sec_can_res} after choosing a particular 
rational scroll $X_0$ containing $\Cc$.

The results there show that the restriction of $\phi$ to the $R(-3)^5$ summand does indeed
have a nontrivial kernel which is the image of the map $(\epsilon_2,-\phi_3)$
in (\ref{Y_res}). It is generated by a single element $m$. It suffices to show that a
non-zero multiple of $m$ cannot lie in  a direct summand of $R(-3)^5$ of rank less than 3.

Take an $R$-basis for $R(-3)^5$ such that the first three elements are a basis for the $R(-3)^3$ summand
of the $R(-3)^3\oplus R(-3)^2$ direct sum decomposition appearing in (\ref{Y_res}). Then $m$ is given by
a 5-vector with quadric entries, the first three of which are the image of the generator of $R(-5)$ under 
$\epsilon_2$.

$\epsilon_2$ is the last map of the twisted Buchsbaum-Eisenbud complex $\Cbe^1(-2)$.
The explicit description of these complexes (see (1.5) \cite{Sch86}, for example) shows that
the first three entries in the vector for $m$ can be taken as the 2x2 minors of the 2x3 matrix
which define $X_0$. These three quadrics are linearly independent in the space of all
quadrics on $\Prj^5$. However, if $am$ lay in a direct summand of $R(-3)^5$ of rank 2
 for a non-zero homogeneous polynomial $a$, the vector for $am$ would be equal to 
$a_1\bf{v}_1+a_2\bf{v}_2$ for homogenous polynomials $a_1,a_2$ and vectors $\bf{v}_1,\bf{v}_2$ 
with entries in $k$ which generate the summand. Then we would arrive at
the contradiction that the five entries of the $am$ vector lay in the space of polynomials
spanned by $a_1$ and $a_2$, which has dimension at most two. 
\end{proof}
\bigskip

A rank $f$ direct summand of $R(e)^d$ is generated by $f$ linearly-independent $d$-vectors 
with coefficients in the base field $k$. The idea is to simply represent $F$ and $G$ as
summands generated by the rows of variable matrices. The condition that $\phi$ maps $G$ into
$F$ is then easily translatable into a system of (degree 1 and 2) polynomial equations in
these variables. The solutions of these over $\bar{k}$ will give us precisely the scrolls $X$
that we are interested in.

Before doing this however, we compute the rank 5 direct summand $F_5$ of $R(-2)^6$ that 
$R(-3)^5$ maps into under $\phi$ and the surface $Y$. These are used in two ways.
\begin{itemize}
\item When $Y$ is an elliptic cone, we will use projection to the genus one curve to construct
$g^1_4$s.
\item $F$ must lie in $F_5$ and it reduces the number of variables in the problem to consider
$F$ as a direct summand of $F_5$ rather than $R(-2)^6$. 
\end{itemize}
We will prove in Section~\ref{pln_mods} that then there are only finitely many $g^1_4$s on
$C$ when $Y$ is not an elliptic cone. Therefore our system of polynomial equations for the scrolls
will be zero-dimensional in that case.
\bigskip

\subsubsection{Computing $F_5$ and $Y$}
These come from basic linear algebra. Let $M_\phi$ be the 5x6 matrix of linear forms giving the
map $\phi$ restricted to $R(-3)^5$. $F_5$ is generated by 5 linearly independent 6-vectors ${\bf v}_i$
with coefficients in $k$ such that each row of $M_\phi$ is a $R$-linear combination of the ${\bf v}_i$.
The one-dimensional right $k$-kernel of the 5x6 matrix with rows ${\bf v}_i$ is precisely
the one-dimensional right $k$-kernel of $M_\phi$. This kernel has to exist and be one-dimensional for
the image of $\phi$ to lie in a direct summand of rank 5 of $R(-2)^6$ but to lie in no smaller
rank direct summand. $Y$ is defined by the 5 quadrics that are the images of the ${\bf v}_i$ under
$\psi$. Thus, we have the following simple algorithm, Algorithm 1, that also gives the matrix $M_5$
for the restricted map $\phi$ from $R(-3)^5$ into $F_5$.

\begin{algorithm}
\caption{Computation of $F_5$, $Y$ and $M_5$}

\begin{description}
\item[Step 0] Input $M_\phi$.
\item[Step 1] Let $M_1$ be the 30x6 $k$-matrix obtained from $M_\phi$ by replacing each linear form
by a 6-element column vector containing its coefficients with respect to the variables of $R$.
Compute $\bf{v}$, a basis for the right kernel of $M_1$.
\item[Step 2] Compute an echelonised basis $B$ for the 5-dimensional $k$-subspace of $k^6$
orthogonal (with respect to the usual scalar product) to $\bf{v}$. $F_5$ is the direct summand
of $R(-2)^6$ with $R$-basis given by the vectors in $B$. Set $I_Y$ equal to the ideal of $R$ generated
by the five quadrics $\psi(b)$ for $b \in B$. $I_Y$ is the defining ideal of $Y$.
\item[Step 3] Let $i$ be a complementary index for the echelonising of basis $B$. That is, if we
remove column $i$ from the 5x6 matrix whose rows are the elements of $B$, we get the identity matrix.
Then, set $M_5$ equal to the 5x5 matrix given by removing the $i$th column of $M_\phi$.
\end{description}

\end{algorithm}
\bigskip

\subsubsection{The elliptic cone case}
\label{ell_cone_case}
Algorithm 2 determines whether $Y$ is a cone over a projectively normal genus 1 curve
$E$ and returns $E$ and the degree 2 projection from $\Cc$ to $E$ in the affirmative case.

Computing the singular locus $S$ of $Y$ is a standard procedure. We apply the Jacobi criterion:
$S$ being defined by the 3x3 minors of the 5x6 Jacobi matrix of partial derivatives of the 5
quadric generators of $I_Y$. A Gr\"obner basis computation tells us whether $S$ is zero-dimensional and
consists of one point or not. Note that the singular locus coincides with the locus of non-smoothness
over $k$ for the possible $Y$ here.

To find a degree 4 map of $\Cc$ to $\Prj^1$ over $\bar{k}$, it remains to compute a degree 2 rational 
function on $E$. One way to do this computationally is to choose any degree 2 effective divisor $D$ of
$E$ over $\bar{k}$ (e.g., twice a $\bar{k}$ point) and compute the Riemann-Roch space $L(D)$ using Florian
Hess' function field package or something similar. Any non-constant function in $L(D)$ is a degree 2
function on $E$. Whether a degree 2 $k$-rational function exists or not on $E$ is trickier.
$E$ is a principal homogeneous space over its elliptic curve Jacobian $J(E)$ of order 1 or 5 in the
Tate-Shaferevich group of $J(E)$ over $k$. If it is of order 5, then the only $k$-rational divisors
on $E$ have degree divisible by 5 and we must go to at an extension of degree at least 5 over $k$ to
find effective divisors of degree 2. If it is of order 1, then $E$ is isomorphic to $J(E)$ over $k$
and there is at least one $k$-rational point. Over $\Q$, there is no known effective procedure for
determining whether a $k$-rational point exists, though in practise, if there are $k$-rational points,
we can often find one by a point search over small height projective points.
\medskip

\noindent{\it Remarks:} 1) It is easy to see that any $g^1_4$ on $\Cc$ is the pullback of a 
$g^1_2$ on $E$. In fact, (6.4) and (6.5) of \cite{Sch86} show that the rulings on a scroll $X$
corresponding to the $g^1_4$ intersect the cone $Y$ in the union of two lines.

2) $E$ is the unique genus 1 curve up to $k$-isomorphism which has a degree 2 covering by
$C$ over $k$. If $E_1$ were another such, consider the product map of $C$ into $E\times E_1$
and let $C_1$ be its image. $C \rightarrow C_1$ is of degree 1 or 2. If it were degree 1, then
$C_1$ would be birationally equivalent to $C$ and an irreducible divisor of $E\times E_1$
of arithmetic genus $\ge 6$ with a degree 2 projection to $E$ and $E_1$. Divisor theory on
$E\times E_1$ shows that this is impossible. Thus $C \rightarrow C_1$ is of degree 2 and
$C_1$ projects $k$-isomorphically onto both $E$ and $E_1$.

\begin{algorithm}
\caption{Testing for the elliptic cone case}

\begin{description}
\item[Step 0] Input $I_Y$.
\item[Step 1] Compute the singular locus $S$ of $Y$. If $S$ is empty or contains more
than one point, return false. Otherwise, let $P$ be the unique ($k$-rational) point in $S$.
\item[Step 2] Find a linear change of coordinates of $\Prj_5$ such that $P$ becomes the point
$(0:0:0:0:0:1)$. If $x_1,\ldots,x_6$ are the new homogeneous coordinates functions, set $B$ equal to
the set of 5 quadrics given by expressing the 5 generators of $I_Y$ in the $x_i$.
\item[Step 3] If $B$ contains an element that is not a quadric in $x_1,\ldots,x_5$ only, return false.
Otherwise, identifying $\Prj^4$ with the hyperplane of $\Prj^5$ defined by $x_6=0$, set $E$ equal to the 
subvariety of $\Prj^4$ with defining ideal generated by $B$ and {\tt prj} equal to the map from $\Cc$
to $E$ given by composing the linear transformation from Step 2 with the projection from $\Prj^5$ to
$\Prj^4$. Return true,$E$,{\tt prj}. 
\end{description}

\end{algorithm}

\bigskip

\subsubsection{The general case}
We now consider the general case when $Y$ is a Del Pezzo. In this case, as mentioned earlier, there are
only finitely many $g^1_4$s and scrolls $X$ (at most 5, in fact). We represent direct summands 
$F$ of $F_5$ and $G$ of $R(-3)^5$ by rows of variable matrices giving their generators and translate
the condition that $\phi$ maps $G$ into $F$ into a system of polynomial equations in the variables. 

We can work projectively but it is computationally simpler
(and involves fewer variables) to work on affine patches of the Grassmannian representations
of the $F$, $G$ summands. For example, one case is where the first minor of each of the matrices
giving the generators of $F$ and $G$ is non-zero. Then we can choose unique linear combinations
of the generators of the two submodules such that $F$ (resp. $G$) is generated by the rows of 
the matrix
$$ M_F = \left(\begin{array}{ccccc}1 & 0 & 0 & u_1 & u_2\\0 & 1 & 0 & v_1 & v_2\\
0 & 0 & 1 & w_1 & w_2\end{array}\right) $$
resp.
$$ M_G = \left(\begin{array}{ccccc}1 & 0 & r_1 & r_2 & r_3\\0 & 1 & s_1 & s_2 & s_3
    \end{array}\right) $$
where $u_i,v_i,w_i,r_i,s_i$ will lie in $\bar{k}$. Considering these as 12 independent
variables, the condition that $G$ maps into $F$ produces a zero-dimensional ideal of relations
$J$ in $k[u_i,v_i,w_i,r_i,s_i]$. The procedure
for computing solutions when the generators of $F$ and $G$ can be put into the 
above form  is given explicitly in Algorithm 3 below. The input is the matrix 5x5 matrix of linear forms 
$M_5$ that represents $\phi: R(-3)^5 \rightarrow F_5$ and was computed earlier in Algorithm 1.

For each set of return values $F_{sol},G_{sol},T_{sol}$ of Algorithm 3, $R$-bases for $F$ as a submodule 
of $F_5$ (resp. $G$ as a submodule of $R(-3)^5$) are given by the rows of $F_{sol}$ (resp. $G_{sol}$)
and the 2x3 matrix representing $\phi:G \rightarrow F$ with respect to these bases is $T_{sol}$.

The corresponding scroll $X$ is defined by the three quadrics that are the images in $R$ of the basis of
$F$ under the inclusion map of $F_5$ into $R(-2)^6$ followed by $\psi$, but equivalently it is
defined by the three quadrics which are the 2x2 minors of $T_{sol}$ (see proof of Prop.~\ref{FG_prop}).

$T_{sol}$ is a 1-generic matrix of homogeneous linear forms in $x_i$
$$\left(\begin{array}{ccc} L_0\ & L_2\ & L_4\\ L_1\ & L_3\ & L_5\end{array}
    \right) $$
As $X$ is defined by the maximal minors of $T_{sol}$, the rational function $L_0/L_1 = L_2/L_3 =
L_4/L_5$ on $X$ gives a rational map to $\Prj^1$ which induces the linear pencil of the ruling.
Since the associated $g^1_4$ on $\Cc$ is the restriction of this ruling, we can therefore read
off the desired degree 4 rational function directly from $T_{sol}$:
\smallskip

\fbox{\begin{minipage}[t]{15cm}\begin{center}
 A degree 4 rational function on $\Cc$ that induces the $g^1_4$  corresponding to $X$\\
  is just $L_0/L_1$  (or $L_2/L_3$, $L_4/L_5$) restricted to $\Cc$\end{center}\end{minipage}}
\begin{algorithm}
\caption{Computation of scrolls $X$ for the first pair of Grassmannian affine patches}

\begin{description}
\item[Step 0] Input $M_5$.
\item[Step 1] Compute $M_G*M_5$. The result is of the form
 $$ \left(\begin{array}{ccccc}\mu_1 & \mu_2 & \mu_3 & \mu_4 & \mu_5\\
  \nu_1 & \nu_2 & \nu_3 & \nu_4 & \nu_5\end{array}\right) $$
  The $\mu_i$ and $\nu_i$ are homogeneous linear forms in the $x_i$ with coefficients
  non-homogeneous linear forms in $r_i$ (for the $\mu_i$) and $s_i$ (for the $\nu_i$).
  The rows of this matrix generate the image of $G$ under $\phi$.
\item[Step 2] The condition that $\phi$ maps $G$ into $F$ translates into 24 equations of degree
  $\le 2$ in the 12 variables given by taking the 4 equations
  $$\mu_4=u_1\mu_1+v_1\mu_2+w_1\mu_3\quad \mu_5=u_2\mu_1+v_2\mu_2+w_2\mu_3$$
  $$\nu_4=u_1\nu_1+v_1\nu_2+w_1\nu_3\quad \nu_5=u_2\nu_1+v_2\nu_2+w_2\nu_3$$
  between linear forms in $x_1,\ldots,x_6$ and equating the coefficients of each $x_i$
  on the LHS and RHS. Let $J$ denote the ideal generated by the 24 equations.
\item[Step 3] Compute a lex Gr\"obner basis of the zero-dimensional ideal $J$ .
  From this, we can read off the solutions in $\bar{k}$ for the $u_i,\ldots,s_i$ to the
  system of 24 equations. 
  For each solution, set $F_{sol}$ and $G_{sol}$ equal to $M_F$ and $M_G$ evaluated at the
  $u_i,\ldots,s_i$ values and set $T_{sol}$ equal to the 2x3 matrix of linear forms in
  $\bar{k}[x_1,\ldots,x_6]$ given by evaluating
  $$\left(\begin{array}{ccc} \mu_1\ & \mu_2\ & \mu_3\\ \nu_1\ & \nu_2\ & \nu_3\end{array}
    \right) $$
  at the $r_i$, $s_i$ values. Return $F_{sol},G_{sol},T_{sol}$.  
\end{description}

\end{algorithm}
\medskip

We have shown how to compute scrolls $X$ and the associated $g^1_4$s and rational maps
that come from $F$ and $G$ summands with a particular affine Grassmannian representation.
The same procedure works in the other cases (e.g., $M_G = 
\left(\begin{array}{ccccc}1 & r_1 & r_2 & 0 & r_3\\0 & s_1 & s_2 & 1 & s_3\end{array}\right)$).
There are $\binom{5}{2}*\binom{5}{3}=100$ possibilities in total, though many will lead to the
same $F,G$ pairs in general.

It follows from \cite{KL74} that a gonality 4, genus 6 curve with only finitely many $g_1^4$s 
has at most 5 $g^1_4$s. We will also prove this in Prop. ~\ref{pln_prop_1},
Section~\ref{pln_mods}. This means that the system of equations that we have to solve
in any particular case give a zero-dimensional ideal $J$ of degree at most 5 generated by
polynomials of degree at most 2. Thus,computing a
lex Gr\"obner of $J$ and even decomposing into prime components is generally quite fast even though
there are 12 variables.

We can adopt the following overall procedure. If we are only interested in finding a degree 4 function
over $\bar{k}$, we can search for solutions over affine patches until we find one and then stop.
If we would like a function over $k$ (if one exists), we can proceed through solutions over different
patches, stopping if we find a $k$-rational one or if we have already found five. If we want to find
all solutions, we can keep searching through patches, stopping if we have found five solutions at any
point. In practise, we compute the prime components of the ideals $J$ while searching and only specialise
into some field extension of $k$ when we have chosen our desired solution.
\bigskip

\subsubsection{Degree six plane models}
We will show in Section~\ref{pln_mods} that $C$ is not birationally equivalent to a degree 6
plane curve in the special case where $Y$ is an elliptic cone.

In the general case, on the other hand, a nice feature of the above computation of scrolls
is that a degree 6 singular plane model for $C$ is also easily constructed from the
$T_{sol}$ matrix.

Note that if $|D|$ is a $g^1_4$ then $|K-D|$ is a $g^2_6$ and vice-versa by Riemann-Roch.
Explicitly, $L0$ and $L1$ in $T_{sol}$ are two hyperplanes that intersect $\Cc$ in degree 10
divisors with a degree
6 common factor $K-D$ where $D$ is the divisor of points lying on $L0$ but not $L1$.
The sections of $|K-D|$ as subsections of $|K|$ correspond to the space of hyperplanes through $D$.
This space is precisely the degree 1 part of the saturation ideal
$(I_\Cc+\langle L0 \rangle : \langle L1 \rangle^\infty)$. Standard Gr\"obner basis
algorithms to compute saturation ideals are well-known (see \cite{PfGr02}).
A 3-element basis of the space defines the map from $\Cc$ to $\Prj^2$ which is
birational onto a degree 6 image $C_1$.

The fact that the map corresponding to $|K-D|$ is birational follows easily from the proof of
Prop.~\ref{pln_prop_1} in Section~\ref{pln_mods}, where the $g^1_4$s on $\Cc$ are identified
with the restrictions of particular linear pencils on the Del Pezzo $Y$. Furthermore, we can
also deduce from the analysis there that $|D|$ corresponds to the pencil of conics in $\Prj^2$
that pass through the 4 singular points of $C_1$ when $C_1$ only has nodes and simple cusps as 
singularities.

When $Y$ is an elliptic cone over genus one curve $E$, the $|K-D|$ for $g^1_4$s $D$ must give
degree 2 maps onto nonsingular degree 3 plane curves (if the image were singular or the map
3-1 onto a plane conic then $C$ would have gonality at most 3). These images must be isomorphic
to $E$ after Remark 2 of \ref{ell_cone_case}.
\bigskip

\subsubsection{Example: $X_0(58)$}
\bigskip

As an example, we take the case of the genus 6 modular curve $X_0(58)$.

With projective coordinate variables $x,y,z,u,v,w$ a canonical model has defining
equations given by the following polynomials

\[\begin{array}{l}
x^2 - xz + yu + yv - xw + uw + vw,\\
x^2 - xy - y^2 - xz - xu + zu + yv + zv,\\
-x^2 + xy + y^2 + u^2 + uv - yw,\\
-x^2 - xy + y^2 + yz - xu + yu + zu + u^2 - vw,\\
y^2 - xz + yz + z^2 - xu + zu + yw + zw,\\
xy - yz - xu + zu + yv + zv + uv - xw + yw + zw + uw\\
\end{array}\]

\noindent This can be derived by computing relations between the $q$-expansions of a
basis of weight 2 cusp forms. The canonical minimal resolution is of general type.
\medskip

Working with \magma\ (\cite{Mag97}) on a 2.2 GHz dual core AMD Opteron machine, it took less than a second
in total to compute the Gr\"obner basis and degree of each of the $J$s corresponding to
different Grassmannian affine patches. As expected there were 5 points in total.
Decomposing a suitable $J$, we found that there was 1 rational solution and 2 pairs of
conjugate solutions over quadratic fields.
\smallskip

Taking the rational solution, it took a fraction of a second to compute that the 
$g^1_4$ is determined by the linear pencil generated by the hyperplanes
$$ 2y - v + w \qquad \mbox{and}\qquad 2x + 2z - v + w $$
and that the complementary $g^2_6$ is determined by the 3 hyperplanes
$$ y + u + v \qquad z - v + w \qquad x - u - v $$
which gives a degree 6 plane model with 4 nodes. The affine version of this has defining
polynomial
$$ \begin{array}{c}Y^5-Y^4X^2-2Y^4X-10Y^4+Y^3X^3+11Y^3X^2+9Y^3X+33Y^3-
    Y^2X^4-2Y^2X^3-\\36Y^2X^2-10Y^2X-37Y^2+3YX^5+22YX^4+
    12YX^3+36YX^2+6YX+17Y-\\2X^6+2X^5-3X^4+4X^3-3X^2+2X-2=0\end{array} $$
where $X$ and $Y$ are the rational functions $(y+v+w)/(x-u-v)$ and $(z-v+w)/(x-u-v)$
on $\Cc$. 
\smallskip

We will now derive radical expressions for $X$ and $Y$ using the degree 4 rational
function on this affine plane curve that comes from the $g^1_4$ on $\Cc$.

As noted above, the $g^1_4$ on the projective plane model is given by the pencil of
quadrics through the 4 nodes. This was simple to compute in \magma. The result is
that the degree 4 function $t$ is given by
$$ t = (3X^2 + Y^2 - 6Y + 3)/(XY - Y^2 + 3X + 4Y) $$

Eliminating $Y$ by a resultant computation, we find that the function field is generated
by $t$ and $X$ and that $X$ satisfies the degree 4 relation over $\Q(t)$
$$ \begin{array}{c} (8t^3+12t^2+8t+4)X^4-(4t^5+6t^4-12t^3-
32t^2-22t-12)X^3-(12t^5+17t^4-4t^3-40t^2-\\
27t-14)X^2-(13t^5+9t^4-9t^3-48t^2-34t-12)X-4t^5
+3t^4+13t^3+32t^2+25t+10 = 0\end{array} $$
A further Gr\"obner basis computation shows that
$$ Y = (2t^2-3t+3)X^2+(t^2-6t)X+2t^2-3t+3)/((2t^2+5t+3)X+5t^2+5t+9) $$

It remains to use the formulae for the roots of a degree 4 polynomial to give a radical expression
for $X$ in $t$. This is rather messy but we include it for completeness. Let $\zeta$ be a primitive
cube root of unity. Define
\begin{eqnarray*} P_1 & = & (3t^5+10t^4+31t^3+40t^2+27t+14)
(72t^{10}+42t^9+475t^8+662t^7+1770t^6+\\
&&2420t^5+2654t^4+2092t^3+
1093t^2+324t+52)\\
P_2 & = & (t^2+2)(16t^4+26t^3+79t^2+42t+45)(t^6+2t^5+11t^4+22t^3+21t^2+12t+4)\\
P_3 & = & 3(7t^6-2t^5-7t^4-14t^3-17t^2-12t-4)\\
A & = & 12t^5+17t^4-4t^3-40t^2-27t-14\\
B & = & 12t^{10}-66t^9-223t^8-926t^7-1974t^6-3332t^5-3854t^4-3340t^3-\\
&&2101t^2-900t-244\\
C & = & t^5+(3/2)t^4-3t^3-8t^2-(11/2)t-3\\ 
D & = & (t+1)(2t^2+t+1)
\end{eqnarray*}
and let $R_1$ be the radical expression
$$ R_1 = \sqrt[3]{-P_1+3P_2\sqrt{P_3}} $$
and
\begin{eqnarray*}
a_1 & = & (-2A +(R_1-(B/R_1)))/(3D)\\
a_2 & = & (-2A -(1/2)((R_1-(B/R_1))+\zeta(R_1+(B/R_1))))/(3D)\\
a_3 & = & (-2A -(1/2)((R_1-(B/R_1))-\zeta(R_1+(B/R_1))))/(3D)
\end{eqnarray*}
Then, we get the following radical expression for $X$
$$ X = (1/4)((C/D)+\sqrt{(C/D)^2-a_1}+\sqrt{(C/D)^2-a_2}+\sqrt{(C/D)^2-a_3})$$
where the third square root should actually be written as a certain rational function in
$t$ divided by the product of the other two square roots.
\bigskip
 
\subsection{Plane Models of Genus 6 Curves}
\label{pln_mods}
\bigskip

In this section, we prove some results about which types of genus 6 curve have degree 6 singular plane 
models and relate the existence of such a model to the finiteness of the number of distinct $g^1_4$s in
the gonality 4 case. As well as having a bearing on our algorithm, these results also give some
interesting information about the stratification of the moduli space of genus 6 curves into
the following parts: hyperelliptic, gonality 3, plane quintic, Clifford-index 2 with $Y$
an elliptic cone and Clifford-index 2 with $Y$ a Del Pezzo. $Y$ refers to the surface described in
Section~\ref{sec_can_res} that contains the canonical curve.

In order to discuss double points more easily, we assume that the characteristic of $k$ is not 2.

\begin{proposition}
\label{pln_prop_1}
For any genus 6 curve $C$ the following are equivalent
\begin{enumerate}
\item[(a)] $C$ is of gonality 4 with no $g^2_5$ and $Y$ is a Del Pezzo surface.
\item[(b)] $C$ has only finitely many $g^1_4$s.
\item[(c)] $C$ is birational over $\bar{k}$ to a degree 6 plane curve with only double points.
\end{enumerate}
In this case, $C$ has at most 5 $g^1_4$s.
\end{proposition}

\begin{proof}
\noindent (b)$\Rightarrow$(a)

\noindent We show that there are infinitely many $g^1_4$s in every other case.

Firstly, assume $C$ has gonality at most 3. We choose a $g^1_3$ containing the effective divisor
$D$. Then for each $P$ in $C(\bar{k})$, $\mbox{dim}|D+P| \ge 1$, so we can find a $g^1_4$ containing
$D+P$. These are all distinct as $P$ is not linearly equivalent to $Q$ for $P\ne Q$, so $D+P$ is
also not equivalent to $D+Q$.

If $C$ is isomorphic to a non-singular degree 5 plane curve ($C$ has a $g^2_5$), then the linear
system of lines through each $P\in C(\bar{k})$ gives a $g^1_4$ after removing the base point $P$.
These are clearly distinct for different $P$.

If $C$ has gonality 4 and $Y$ is a cone over an elliptic curve $E$, then the pullback under the
degree 2 projection of $C$ to $E$ of the infinitely many distinct $g^1_2$s on $E$ lead to 
infinitely many distinct $g^1_4$s on $C$ (the pullbacks of $|P+Q|$ and $|R+S|$ can be the same on
$C$ only if $P+Q=R+S$ in the group law of $E$).
\medskip

\noindent (a)$\Rightarrow$(c)

\noindent As noted before, the inverse of a birational parametrisation $\Prj^2 \rightarrow Y$ by cubic
polynomials (over $\bar{k}$) maps $\Cc$ birationally to a plane sextic $C_1$. If $C_1$ had
a singularity of multiplicity at least 3, projection from it would give lead to a map from
$C$ to $\Prj^1$ of degree at most 3, contradicting the gonality 4 assumption. Thus, $C_1$
only has double points.
\medskip

\noindent (c)$\Rightarrow$(b)

\noindent Let $C_1$ be a plane sextic birational to $C$ with only double points. Since we are
assuming that $\mbox{char}(k) \ne 2$, any singularity is of type $A_n$ for $n \ge 1$
(analytically isomorphic to $y^2=x^{n+1}+\mbox{higher powers of}\ x$). Such a singulatity
is resolved by sequence of $\lfloor (n+1)/2\rfloor$ blowings up: at each stage, if we have
an $A_n$ singularity $P$ on the transformed curve in the blown-up plane, after blowing up at
$P$, the strict transform of the curve has a single $A_{n-2}$ singularity over $P$ if $n>2$,
a single nonsingular point over $P$ if $n=2$ (simple cusp), or 2 nonsingular points over $P$
if $n=1$ (node). We continue getting singularities of multiplicity 2 until the singularity is
resolved. By \cite{H77}, the arithmetic genus of the strict transform of $C_1$ drops by one
at each blow up until it reaches 6, when it must be isomorphic to $C$ and non-singular. Thus we have
to blow up exactly 4 times and then the blow-up of $\Prj^2$ is $Y_1$, a degree 5 Del Pezzo
surface. $Y_1$ is a degenerate Del Pezzo if there is an $A_n$ singularity with $n > 2$, when
we have to blow up infinitely near points. Note that for all possible singularity types of
$C_1$ ($4(A_1/A_2)$, $2(A_1/A_2)+1(A_3/A_4)$, $2(A_3/A_4)$, $1(A_1/A_2)+1(A_5/A_6)$,
$1(A_7/A_8)$), the sequence of blow-up points is "almost general" as per Definition 1, Section III.2
of \cite{Dem76}, so $Y_1$ is a (possibly degenerate) Del Pezzo. In particular, if a line went through
all 4 blow-up points (including infinitely near ones), then its intersection number with $C_1$ would
be at least 8, which is impossible as $C_1$ is irreducible of degree 6.
The strict transform of $C_1$ is isomorphic to $C$, so we will identify it with $C$.

Let $H$ denote the divisor class of $Y_1$ that is the total transform of the class of a
line in $\Prj^2$ and $E_1,E_2,E_3,E_4$ the total transforms of the exceptional curves
generated in the blow-ups (so $E_i^2=1$ and $E_i\cdot E_j=0$ for $i\ne j$).
$C$ is an irreducible divisor of $Y_1$ rationally equivalent to $6H-2E_1-2E_2-2E_3-2E_4=-2K_{Y_1}$
where $K_{Y_1}$ is the canonical class on $Y_1$. By adjunction
(Prop. 8.20, Chap. II, \cite{H77}), the
canonical class $K_C$ on $C$ is the restriction of the class $-K_{Y_1}$ of $Y_1$. Also as
$H^1(Y_1,O(K_{Y_1}))=0$ ($Y_1$ is rational), the standard cohomology sequence shows that
the restriction map on global sections $H^0(Y_1,O_{Y_1}(-K_{Y_1})) \rightarrow H^0(C,O_C(K_C))$
is an isomorphism. Thus, the anticanonical map of $Y_1$ into $\Prj^5$ restricts to the canonical
embedding of $C$ onto $\Cc$. Note that $-K_{Y_1}$ is birational on $Y_{1}$ and $1-1$ outside of
some possible (-2)-curves, so the canonical map on $C$ is birational. Hence, $C$ is not 
hyperelliptic and its canonical map gives an embedding.
Let the image of $Y_1$ be $Y_0 \supset \Cc$. Note that $Y_0$ is
singular when $Y_1$ is degenerate: fundamental cycles of (-2)-curves are contracted to (simple)
singular points (see \cite{Dem76}).

Now, the one-dimensional linear system $|D|$ on $Y_1$, where $D=2H-E_1-E_2-E_3-E_4$, restricts to
a complete $g^1_4$ on $C$. It consists of the strict transforms on the pencil of conics 
through the 4 blow-up points (if there are repeated blowups, some of the divisors in
the pencil may be the sum of the strict transform of a conic and some irreducible
(-1)-curves and/or (-2)-curves). It is of degree 4 on $C$ as
$D\cdot(-2K_{Y_1}) = 4$. It gives a {\it complete} $g^1_4$ because $-D-K_{Y_1}=H$ and so, by Serre
duality, $h^1(Y_1,O_{Y_1}(D-C))=h^1(Y_1,O_{Y_1}(D+2K_{Y_1}))=h^1(Y_1,O_{Y_1}(H))=0$. The last 
equality
holds by applying Riemann-Roch to $H$ and using $h^0(H)=3$, $h^2(H)=h^0(K_{Y_1}-H)=0$. It isn't 
hard
to see that the intersection of all of the divisors in $|D|$ is either empty or consists of a union of
(-2)-curves. However, $C$ is irreducible and $C\cdot E = (-2K_{Y_1})\cdot E = 0$, so $C$ doesn't
intersect any (-2)-curve $E$. Thus, the $g^1_4$ is also basepoint free.

This $g^1_4$ corresponds to the rulings of a degree 3, dimension
3 scroll $X$ containing $\Cc$ (this is true for any complete, basepoint free $g^1_4$). Further, Schreyer's
derivation of the canonical resolution using $X$ applies to give the description in section 4.1
(see sections 4 and 6.2 of \cite{Sch86}) and show that the minimal free resolution of $R/I_\Cc$ is of
shape \res. This means that $C$ must be of gonality 4 with no $g^2_5$.

We now show that the $Y$ associated to $X$ is precisely $Y_0$.

Firstly, $Y_0\subset X$. The two-dimensional linear spaces comprising the ruling of $X$ are the
linear spans of the supports of the divisors in the $g^1_4$. However, each such divisor is the
intersection of $\Cc$ with the image of a divisor $E$ in $|D|$. Since, $D\cdot(-K_{Y_1})=2$, these images
are of degree 2 and so lie in a plane which must be the span of the corresponding divisor on $\Cc$
except for a finite number of degenerate cases. As the divisors in $|D|$ cover $Y_1$, its image lies in
$X$.

Now $X$ is a scroll of type $S(1,1,1)$ or $S(2,1,0)$ in Schreyer's notation from \cite{Sch86}. The first
is non-singular and the second is a cone with a single singular point $P$. In the second case, if
$Y_0$ contained $P$ then it would be a singular point of $Y_0$. This is because $P$ is the intersection
of all of the ruling planes and so would lie in the intersection of all of the images of divisors in
$|D|$. As noted above, this intersection, if non-empty, consists of some irreducible 
(-2)-curves, so must map into singular points under $-K_{Y_1}$. Thus, in either case, $Y_1 \rightarrow
Y_0 \subset X$ factors through a map $Y_1 \rightarrow \Prj(\mathcal{E})$, where $\mathcal{E}$ is
a 3-dimensional locally-free sheaf on $\Prj^1$ and $\Prj(\mathcal{E})$ is the projective bundle
of which $X$ is the image. We abuse notation slightly by denoting the strict transforms of $Y_0$ and
$Y$ in $\Prj(\mathcal{E})$ by the same symbols. Let $H'$ denote the divisor class of the pullback
of the Cartier hyperplane class of $X$ and $R$ the class of the ruling. By standard theory,
$Pic(\Prj(\mathcal{E}))$ is generated by $H'$ and $R$ (see, e.g., Ex. 12.3, Chapter III of \cite{H77}). We have that
$H'^3 = 3$ (the degree of $X$), $H'^2\cdot R=1$ (as the rulings map to linear spaces in $\Prj^5$) and
$R^2=0$ in the cycle class group of $\Prj(\mathcal{E})$. We know that $Y$ is equivalent to
$2H'-R$. We want to show the same for $Y_0$. Let $Y_0$ be equivalent to $aH'+bR$ with $a,b \in \Z$.
$Y_0$ of degree 5 in $\Prj^5$ implies that $H'^2\cdot Y_0=5$. Also, $R$ pulls back to $D$ in $Y_1$,
which means that $Y_0\cdot R\cdot H' = D.(-K_{Y_1})=2$. Thus we get $a=2$ and $b=-1$ as required.

If $Y \ne Y_0$, since both are irreducible, $U = Y \cap Y_0$ would be a sum (with multiplicities)
of irreducible curves in $\Prj(\mathcal{E})$ with $U\cdot H'=(2H'-R)\cdot (2H'-R)\cdot H'=8$.
But $\Cc\subset U$ and $\Cc\cdot H'=10$. Thus $Y$ is $Y_0$ as claimed. We can also
note that $\Cc$ misses any singular locus of $X$, as it is non-singular and a complete intersection
in $X$.

Now we know that any $g^1_4$ on $C$ comes from a ruling on some $X$ which restricts to a pencil of
divisors on the unique $Y$, which is $Y_0$. This pulls back to a pencil, contained in $|D_1|$ say,
on $Y_1$.
Since the ruling induces a degree 4 divisor on $\Cc$ and $\Cc$ pulls back isomorphically to $C$ which
is equivalent to $-2K_{Y_1}$ on $Y_1$, $D_1\cdot K_{Y_1} = -2$. If $X$ is nonsingular, or $Y_0$
misses its unique singular point $P$, $D_1^2=0$. Otherwise, any two rulings meet transversally at
$P$, which must be singular point of $Y_0$, so $D_1=D_0+F$ where $F$ is the fixed part of the pencil -
a fundamental cycle of (-2)-curves - and $D_0$ is the variable part with 
$D_0^2 = 0$. As $C$ is disjoint from all irreducible (-2)-curves $E$ and $K_{Y_1}\cdot E=0$, we can 
replace $D_1$ by $D_0$ in this case. Thus any $g^1_4$ comes from the restriction of some pencil
of divisors within $|D_1|$ where $D_1^2=0$ and $D_1\cdot K_{Y_1} = -2$. Since a $g^1_4$ on $C$ is
complete, it is entirely determined by the class of any divisor within it, and so any pencil induced
from $D_1$ will only depend on the equivalence class of $D_1$ in $Pic(Y_1)$.

A simple computation shows that the only classes with self intersection 0 and intersection
number 2 with $-K_{Y_1}$ are the four classes $H-E_i$, $1\le i\le 4$, and $D=2H-E_1-E_2-E_3-E_4$
from above. Thus we have the stronger result (as expected from Brill-Noether) that there are at most
five distinct $g^1_4$s on $C$.

\end{proof}
\bigskip

\noindent {\it Remarks:} 1) In some cases, several members of the five classes listed at the end
of the last proof can restrict to the same $g^1_4$ on $C$. This can occur when the plane model
of $C$ has higher order $A_n$ singularities and one of the classes differs from another by
 a fixed component consisting of (-2)-curves [$C \equiv -2K_{Y_1}$ on $Y_1$ means that
(-2)-curves are the only effective divisors which don't intersect $C$].

2) When the plane model $C_1$ only has nodes and simple cusps, the five $g^1_4$s correspond
to the four pencils of lines through a singular point and the pencil of conics through all
four singular points.
\bigskip

\noindent {\it Example:} Let $C$ be birational to the plane curve $C_0$ with defining polynomial
(w.r.t. $x,y,z$ projective coordinates)
$$\begin{array}{c} x^2y^4+2x^3y^2z-xy^4z+ x^4z^2 + 2x^3yz^2 - x^2y^2z^2 - xy^3z^2 - 2y^4z^2 - x^3z^3\\ + 2x^2yz^3 - xy^2z^3 + y^3z^3 - x^2z^4 + 2xyz^4 + y^2z^4 \end{array}$$
$C_0$ has nodes ($A_1$ singularities) at $P_2 := (0:0:1)$ and $P_3 := (0:1:0)$ and a type
$A_3$ singularity at $P_1 := (1:0:0)$. $P_4$ will denote the infinitely near point above $P_1$
that needs to be blown up to resolve the $A_3$ singularity.
We write $L_{ij}$ for the line through $P_i$ and $P_j$
($L_{14}$ is the line through $P_1$ whose tangent direction corresponds to $P_4$).
$L_{12}$ doesn't pass through $P_4$ but $L_{13}$ does, so $L_{13}=L_{14}=L_{34}$
and $L_{24}$ doesn't exist.
Blowing up the points in the order $P_1$,$P_2$,$P_3$,$P_4$, we see that 
$E_1 = \hat{E}_1+E_4$ with $\hat{E}_1$ an irreducible (-2)-curve and
$E_2,E_3,E_4$ irreducible (-1)-curves.

$\hat{L}$, the strict transform of $L_{13}$, is an irreducible (-2)-curve.
Letting $E_{ij}$ denote the usual (-1)-classes $H-E_i-E_j$ on $Y_1$,
$E_{12}$ and $E_{23}$ give irreducible (-1)-curves, but $E_{13}=\hat{L}+E_4$,
$E_{14}=\hat{L}+E_3$, $E_{34}=\hat{L}+\hat{E}_1+E_4$ and $E_{24}=E_{12}+\hat{E}_1$.

We easily find that $|H-E_1|$, $|H-E_2|$ and $|H-E_3|$ are distinct pencils without fixed
points or components on $Y_1$, but $(H-E_4)\cdot\hat{E}_1 = D\cdot\hat{L} = -1$ means that
$\hat{E}_1$ and $\hat{L}$ are fixed components of the two respective pencils and that
$|H-E_4|=|H-E_1|+\hat{E}_1$ and $|D|=|H-E_2|+\hat{L}$. In this case, therefore, the five pencils
only restrict to 3 distinct $g^1_4$s on $C$.
\smallskip

Taking the canonical map $C_0 \rightarrow \Prj^5$ defined by the polynomials
$(x^2z:xy^2:xyz:xz^2:y^2z:yz^2)$ to give $\Cc$, we find explicitly that $Y$ is defined
by the 5 quadrics
$$z^2-xt,\quad ys - xt,\quad zs - xu,\quad zt - yu,\quad st - zu$$
and that there are indeed only 3 scrolls containing $\Cc$:
$$\begin{array}{l} X_1:\qquad ys-xt = zs-xu = zt-yu = 0\\
 X_2:\qquad  z^2-ys=zt-yu=st-zu=0\\
 X_3:\qquad  z^2-xt=zs-xu=st-zu=0  \end{array}$$
$Y$ has singularities at $p_1 := (1:0:0:0:0:0)$ and $p_2 :=(0:1:0:0:0:0)$, the images of
$\hat{E_1}$ and $\hat{L}$. $X_1$ is non-singular, $X_2$ has singular point $p_1$ and 
$X_3$ has singular point $p_2$. This illustrates that the $Xs$ for a given $C$ may be of
different type and that $Y$ may pass through the singular point of an $X$.

\begin{proposition}
\label{pln_prop_2}
A genus 6 curve $C$ has gonality 3 if and only if it is birationally equivalent over $\bar{k}$ to a
degree 6 plane curve with at least one singularity of multiplicity greater than 2.
\end{proposition}

\begin{proof}
Let $C$ be birationally equivalent to $C_1$, a degree 6 plane curve with a singularity $P$ of 
multiplicity of at least 3. Projection from $P$ gives a map to $\Prj^1$ of degree at most
3 and $C$ is hyperelliptic or has gonality 3. Since $C$ has genus 6, there are only two possibilities
for the singularities of $C_1$. Either there is an triple point that is resolved by a single blow-up and a
cusp or node, or there is a single triple point with a single infinitely-near
node or cusp. In either case, it is easy to see that the canonical linear system on $C_1$ , which is
given by the subsystem of plane cubics that satisfy the adjoint conditions at the singularities,
separates the non-singular points. Therefore, $C$ cannot be hyperelliptic. For example, in the first case
we can take reducible cubics consisting of two lines through the triple point and one line through the 
other singularity.

Conversely, let $C$ have gonality three. Then $C$ is isomorphic to a 3-section in a Hirzebruch surface,
which is swept out in the canonical embedding by the line spans of the divisors of a $g^1_3$. There are
two possibilities for the surface and the divisor class of $C$ within it: $X_0$ with $C \sim 3C_0+4f$
or $X_2$ with $C \sim 3C_0+7f$ in the notation of Chapter V, Section 2 of \cite{H77}. In either case,
we can find a birational isomorphism of $X_i$ onto $\Prj^2$ that maps $C$ onto a degree 6 curve.

In the $X_2$ case, let $P$ be any point on $C$ which doesn't lie on the distinguished (-2)-curve $C_0$.
We begin by blowing up $P$ and blowing down the strict transform of the fibre $f$ containing $P$.
This gives a birational isomorpism of $X_2$ to $X_1$ where the unique (-1)-curve is the
strict transform of $C_0$. We then blow down this (-1)-curve to get to $\Prj^2$. Tracing through
the intersections of transforms of $C$ with the curves that are blown down, it is easy to see
that the self-intersection of the overall strict transform of $C$ is 36 ($C^2=24$), so that it is a
degree 6 plane curve.
Furthermore it has a unique singular point that is a triple point (which resolves to the intersection
of $C$ with $f$ on $X_2$) with a single node or cusp above it on $X_1$, so we have the second case
above for singularities.

In the $X_0$ case, we proceed similarly. Here, the class of $C_0$ does not contain a unique curve but
gives a second ruling of $X_0$. The rulings $C_0$ and $f$ are distinguished here by $C\cdot f=3$ and
$C\cdot C_0=4$. We choose any point $P$ on $C$ to blow up but, this time, we will let $C_0$ denote
the unique curve in its class passing through $P$. Then, we blow down the strict transform of $f$
followed by the strict transform of $C_0$ as above. Again we see that the overall strict transform
$C_1$ of $C$ has degree 6 in the plane. If $f$ meets $C$ transversally at $P$, we get a triple point and
a node/cusp on $C_1$, which resolve to the intersection of the strict transforms of $C$ and $C_0$
and the strict transforms of $C$ and $f$ after the blowing up of $P$. $X_0$ is isomorphic to a
non-singular quadric surface in $\Prj^3$ and the birational map to $\Prj^2$ corresponds to projection
from the point $P$ on that model.
\end{proof}

\bibliographystyle{amsalpha}
\bibliography{mike.bib}
\end{document}